\documentclass[final,1p,number,lefttitle,nopreprintline]{elsarticle}
\usepackage[T1]{fontenc}
\usepackage{lmodern}
\usepackage{amsmath,amssymb,amsthm,mathtools}
\usepackage{microtype}
\usepackage{needspace}
\usepackage{enumitem}
\usepackage{setspace}
\usepackage{fancyhdr}
\usepackage[hidelinks]{hyperref}

\numberwithin{equation}{section}
\biboptions{numbers,square,sort&compress}

\AtBeginEnvironment{thebibliography}{\setstretch{1}}

\newtheorem{theorem}{Theorem}[section]
\newtheorem{lemma}[theorem]{Lemma}
\newtheorem{proposition}[theorem]{Proposition}
\newtheorem{corollary}[theorem]{Corollary}
\theoremstyle{definition}
\newtheorem{remark}[theorem]{Remark}
\makeatletter
\renewenvironment{proof}[1][\proofname]{\par
  \pushQED{\qed}\normalfont
  \topsep6\p@\@plus6\p@\relax\trivlist
  \item[\hskip\labelsep\bfseries #1\@addpunct{.}]\ignorespaces
}{\popQED\endtrivlist\@endpefalse}
\makeatother
\newcommand{\Z}{\mathbb Z}
\newcommand{\Q}{\mathbb Q}
\newcommand{\C}{\mathbb C}

\makeatletter
\patchcmd{\MaketitleBox}{\Large\@title}
  {\fontsize{17}{20}\selectfont\@title}{}{
  \PackageError{advances-format}{Unable to format title}{}}
\patchcmd{\MaketitleBox}{\normalsize\elsauthors}
  {\fontsize{12}{14}\selectfont\elsauthors}{}{
  \PackageError{advances-format}{Unable to format authors}{}}
\patchcmd{\MaketitleBox}{\vskip36pt}{\vskip16pt}{}{
  \PackageError{advances-format}{Unable to format frontmatter spacing}{}}
\patchcmd{\emailauthor}{\ttfamily}{\normalfont}{}{
  \PackageError{advances-format}{Unable to format email addresses}{}}
\makeatother

\journal{Advances in Mathematics}
\hypersetup{
 pdfauthor={Hu Tan and Ying Zhang},
 pdftitle={Fuglede's conjecture for cyclic groups of square-free order},
 pdfsubject={Spectral sets and translational tilings in finite cyclic groups},
 pdfkeywords={Fuglede conjecture, spectral set, cyclic group, square-free order, character matrix}
}
\AtBeginDocument{\hypersetup{hidelinks}}

\begin{document}
\begin{frontmatter}
\title{Fuglede's conjecture for cyclic groups of square-free order}
\author[amss]{Hu Tan}
\ead{tanhu2020@amss.ac.cn}
\author[soochow]{Ying Zhang}
\ead{yzhang@suda.edu.cn}
\affiliation[amss]{
 organization={Academy of Mathematics and Systems Science,
 Chinese Academy of Sciences},
 city={Beijing},postcode={100190},country={China}}
\affiliation[soochow]{
 organization={School of Mathematical Sciences, Soochow University},
 city={Suzhou},postcode={215006},country={China}}
\begin{abstract}
We prove Fuglede's conjecture for finite cyclic groups of square-free order. Specifically, if \(N\) is square-free and \(A\subseteq \mathbb Z_N\) is nonempty with \(|A|=m\), then \(A\) is spectral if and only if it tiles \(\mathbb Z_N\); moreover, these conditions hold precisely when \(m\mid N\) and \(A\) is a complete set of representatives modulo \(m\). This gives a complete structural description of spectral and tiling sets in \(\mathbb Z_N\). The key new ingredient is a prime-coordinate elimination theorem: if \(P\nmid nm\) is an odd prime, then every \(m\)-point spectral pair in \(\mathbb Z_n\times\mathbb Z_P\) projects injectively to a spectral pair in \(\mathbb Z_n\). Its proof combines character-matrix orthogonality, cyclotomic coefficient comparison, and a \(P\)-adic divisibility argument for associated Gram matrices. Iterating this reduction yields the result for arbitrary square-free \(N\), with no restriction on the number or relative sizes of its prime factors.
\end{abstract}
\begin{keyword}
Fuglede's conjecture \sep spectral sets \sep cyclic groups \sep
translational tilings \sep cyclotomic polynomials
\MSC[2020]{Primary 05B45; Secondary 43A40, 20K01}
\end{keyword}
\end{frontmatter}
\setstretch{1.13}

\section{Introduction}

Fuglede's conjecture asks whether translational tiling and the existence
of an orthogonal basis of exponentials are equivalent. In its original
measurable formulation, a bounded measurable set
$\Omega\subset\mathbb R^d$ of positive measure is \emph{spectral} if
$L^2(\Omega)$ has an orthogonal basis
$\{e^{2\pi i\langle\lambda,x\rangle}:\lambda\in\Lambda\}$ for some
$\Lambda\subset\mathbb R^d$. It \emph{tiles by translations} if
$\sum_{t\in T}\mathbf 1_\Omega(x-t)=1$ almost everywhere for some
$T\subset\mathbb R^d$. Fuglede proposed the equivalence in 1974 in his
study of commuting self-adjoint extensions of partial differential
operators~\cite{Fuglede}. He proved the lattice case: a set tiles by a
full-rank lattice $L$ if and only if its dual lattice $L^*$ is a spectrum.

For general measurable sets, Tao disproved the spectral-to-tiling
implication in dimensions at least five~\cite{Tao}; Matolcsi reduced
the dimension to four~\cite{Matolcsi4}, and Kolountzakis and Matolcsi
to three~\cite{KM}. Counterexamples to the reverse implication were
constructed in dimension five by Kolountzakis and Matolcsi~\cite{KMTiles},
in dimension four by Farkas and R\'ev\'esz~\cite{FarkasRevesz}, and in
dimension three by Farkas, Matolcsi, and M\'ora~\cite{FMM}. A recent
preprint of Zhang gives counterexamples to both implications in dimension
two~\cite{ZhangPlane}. The unrestricted one-dimensional conjecture remains
open. For convex sets, by contrast, the equivalence was established for
planar domains by Iosevich, Katz, and Tao~\cite{IKT}, for three-dimensional
polytopes by Greenfeld and Lev~\cite{GL}, and for convex bodies in every
dimension by Lev and Matolcsi~\cite{LMConvex}.

\subsection{The discrete and cyclic conjectures}

The discrete formulation replaces Lebesgue measure by counting measure
and exponentials by characters. A nonempty subset $A$ of a finite
abelian group $G$ is spectral if some set of characters of $G$ restricts to
an orthogonal basis of $\ell^2(A)$; it tiles $G$ if the addition map
$A\times B\to G$ is bijective for some $B\subseteq G$. The discrete
conjecture is also false for general finite abelian groups. Indeed,
finite-group examples underlie the Euclidean counterexamples
in~\cite{Tao,KM,KMTiles}. For vector spaces over prime fields, the
conjecture holds in $(\Z_p)^2$ for every prime $p$~\cite{IMP}, whereas
spectral sets that do not tile exist in $(\Z_p)^4$ for every odd
prime $p$~\cite{FergusonSothanaphan}. The ambient group is therefore
essential to the question.

The cyclic case is formulated as follows. Write $\Z_N=\Z/N\Z$.
For $A\subseteq\Z_N$ with $|A|=m>0$, a set
$\Lambda\subseteq\Z_N$ is a spectrum if $|\Lambda|=m$ and
\begin{equation}\label{eq:spectral-matrix}
 F=\bigl(e^{2\pi i a\lambda/N}\bigr)_{a\in A,\lambda\in\Lambda}
 \quad\text{satisfies}\quad FF^*=F^*F=mI_m.
\end{equation}
We call $(A,\Lambda)$ a \emph{spectral pair}, and write
$A\oplus B=\Z_N$ when each element of $\Z_N$ has a unique representation
as $a+b$, with $a\in A$ and $b\in B$. The \emph{finite cyclic Fuglede
conjecture} asserts that, for every positive integer $N$ and every
nonempty $A\subseteq\Z_N$,
\[
 A\text{ is spectral in }\Z_N
 \quad\Longleftrightarrow\quad
 A\oplus B=\Z_N\text{ for some }B\subseteq\Z_N.
\]
The conjecture remains open in this generality, and no counterexample
in a finite cyclic group is known; see~\cite{Malikiosis,Somlai}.

This question is closely related to the remaining measurable problem
on the real line. Dutkay and Lai~\cite[Theorem~1.3]{DL} established
reductions between bounded subsets of $\mathbb R$, finite subsets of
$\Z$, and finite cyclic groups. In particular, the tiling-to-spectral
implication on $\mathbb R$ is equivalent to that implication for all
finite cyclic groups. In the opposite direction, their reduction from
all cyclic groups to $\mathbb R$ requires every bounded spectral set of measure one to admit a rational
spectrum; a recent preprint of Fu and Song proves
this rationality assertion~\cite{FS}. These connections concern the
entire family of cyclic groups, including orders with repeated prime
factors. Known positive results include prime-power
orders~\cite{Laba}, $p^a q$~\cite{MK}, $p^a q^2$~\cite{KMSV20}, and
$p^a qr$~\cite{ZhangRing}, where $a\ge1$ and the displayed primes are distinct.

\subsection{Square-free orders}

We consider the case
\[
 N=p_1\cdots p_k,\qquad p_1,\ldots,p_k\text{ distinct primes}.
\]
For such orders, the tiling-to-spectral implication was already known
without any restriction on $k$. More precisely, every tile of cardinality
$m$ is a complete set of representatives modulo $m$, and therefore has
the subgroup $m\Z_N$ as a tiling complement. This follows from
Tijdeman's dilation theorem~\cite{Tijdeman} in the finite polynomial
form of Coven and Meyerowitz~\cite[Lemma~3.1]{CM}; the square-free
consequence was discussed by Meyerowitz and \L aba in the comments
to~\cite{TaoNotes}, and a published proof appears
in~\cite[Proposition~4.1]{Shi}. Malikiosis established the
tiling-to-spectral implication for the larger family
$N=p_1^a p_2\cdots p_k$~\cite{Malikiosis}. Thus the outstanding
square-free problem was the spectral-to-tiling implication.

For a bounded number of distinct prime factors, the progression is now
well established. The case of one prime is elementary. The case $pq$
is contained in the 2017 theorem of Malikiosis and Kolountzakis for
$p^a q$~\cite{MK}. Shi proved the case of three distinct primes,
$N=pqr$, in 2019~\cite[Theorem~1.3]{Shi}. Kiss, Malikiosis, Somlai, and Vizer then
proved the case of four distinct primes, $N=pqrs$, in 2022~\cite[Theorem~1.4]{KMSV}.

Somlai's 2026 preprint~\cite[Theorem~1.1]{Somlai} provides an induction
beyond a fixed number of prime factors. It proves that, if $n$ is
square-free and $P>n$ is prime, then the spectral-to-tiling implication
for $\Z_n$ entails the same implication for $\Z_{nP}$. Combined with
the four-prime theorem, this covers the five-prime case
\[
 N=pqrsP\qquad\text{whenever }P>pqrs.
\]
Iterating the same result gives square-free orders with arbitrarily
many prime factors, subject to a growth condition: starting with any
four distinct primes, one may successively adjoin primes satisfying
\[
 p_{j+1}>\prod_{i=1}^{j}p_i\qquad(4\le j<k).
\]
His argument develops the large-prime method of Fallon, Kiss, Mayeli,
and Somlai~\cite{FKMS}, whose result concerns
$N=p^2q^2r$ under $r>p^2q^2$, and uses cyclotomic bounds of
\L aba and Marshall~\cite{LabaMarshall}.

\subsection{Main result}

We prove the conjecture for \emph{every} square-free order. Thus the
number $k$ of distinct prime factors is arbitrary, and no inequality
between their sizes is required. The result also identifies all spectral
sets through the same complete-representative property already known
for tiles.

\begin{theorem}\label{thm:main}
Let $N\ge1$ be square-free, and let $\varnothing\ne A\subseteq\Z_N$
have cardinality $m$. The following assertions are equivalent:
\begin{enumerate}[label=\textup{(\arabic*)},leftmargin=2.3em]
\item $A$ is spectral in $\Z_N$;
\item $A$ tiles $\Z_N$;
\item $m\mid N$ and reduction modulo $m$ maps $A$ bijectively onto $\Z_m$.
\end{enumerate}
When these conditions hold,
\[
 A\oplus m\Z_N=\Z_N,
 \qquad
 \frac Nm\{0,1,\ldots,m-1\}
 \quad\text{is a spectrum of }A.
\]
\end{theorem}

The new implication is \textup{(1)}$\Rightarrow$\textup{(3)}: spectrality
alone forces $m\mid N$ and makes the quotient map
$\Z_N\to\Z_m$ bijective on $A$. The equivalence of \textup{(2)} and
\textup{(3)}, together with their implication of \textup{(1)}, is the
established tiling result recalled above. Consequently,
Theorem~\ref{thm:main} extends the known three- and four-prime results
and removes the prime-growth restriction from the previously known
families with arbitrarily many distinct prime factors.

\subsection{Proof strategy}

The main step is Theorem~\ref{thm:eliminate}: an $m$-point spectral
pair in $\Z_n\times\Z_P$ projects injectively to a spectral pair
in $\Z_n$ whenever $P$ is an odd prime not dividing $nm$.
Its proof uses two lemmas. Rows and columns of the character matrix
are grouped by their $P$-coordinates. After removal of the
$P$-coordinate factors, row orthogonality in the original matrix and
cyclotomic coefficient comparison (Lemma~\ref{lem:coefficients}) show that the Gram matrices
of the column groups have a common off-diagonal block part.
Column orthogonality supplies the quadratic identities required by
Lemma~\ref{lem:block-rigidity}, which forces this common part to
vanish. This proves orthogonality of the projected character matrix.
The argument imposes no lower bound on $P$.

Proposition~\ref{prop:spectral-transversal} iterates this elimination
over the odd prime factors of $N$ that do not divide $m$.
A cardinality argument, followed by a parity argument for a possible
remaining coordinate of order two, proves that $m\mid N$ and that
reduction modulo $m$ is bijective on $A$. This establishes the new
implication of Theorem~\ref{thm:main}. The proof is completed in
Section~\ref{sec:mainproof} using the classical finite dilation lemma
for tiles and an explicit subgroup spectrum.

\section{Prime-coordinate elimination}\label{sec:elimination}

Let $\zeta_d=e^{2\pi i/d}$, and let $\Phi_d$ denote the $d$th
cyclotomic polynomial. For a prime $P\nmid n$, we use the product
pairing on $\Z_n\times\Z_P$:
\[
 \chi_{(u,j)}(x,i)=\zeta_n^{xu}\zeta_P^{ij}.
\]
The spatial Chinese remainder map is
$a\mapsto(a\bmod n,a\bmod P)$. Under this map, the standard
frequency $\lambda\in\Z_{nP}$ has product coordinates
\[
 (P^{-1}\lambda\bmod n,\ n^{-1}\lambda\bmod P).
\]
The inverses are taken in the indicated coordinates, with the
first coordinate trivial when $n=1$. Thus the product pairing
relabels the dual group, while deleting the spatial $P$
coordinate is the natural quotient onto $\Z_n$.

We first record the coefficient comparison used below. The same
argument appears in the proof of~\cite[Lemma~2.13]{KMSV}.

\begin{lemma}[Cyclotomic coefficient comparison]\label{lem:coefficients}
Let $n\ge1$, let $P$ be a prime not dividing $n$, and let
$c\in\Z_P\setminus\{0\}$. If $b_0,\ldots,b_{P-1}\in\Q(\zeta_n)$
satisfy
\[
 \sum_{j=0}^{P-1}b_j\zeta_P^{cj}=0,
\]
then $b_0=\cdots=b_{P-1}$.
\end{lemma}
\begin{proof}
Since $P\nmid n$, the cyclotomic degree formula gives
\[
 [\Q(\zeta_n,\zeta_P):\Q(\zeta_n)]
 =\frac{\varphi(nP)}{\varphi(n)}=P-1.
\]
Thus $\Phi_P(X)=1+X+\cdots+X^{P-1}$ is the minimal polynomial
of the primitive $P$th root $\zeta_P^c$ over $\Q(\zeta_n)$.
The polynomial $\sum_{j=0}^{P-1}b_jX^j$ has degree at most $P-1$
and vanishes at $\zeta_P^c$, so it is a scalar multiple of $\Phi_P$,
possibly the zero polynomial. Its coefficients are therefore equal.
\end{proof}

The next lemma isolates the matrix argument. No commutativity
between the matrices is assumed.

\begin{lemma}[Block-matrix rigidity]\label{lem:block-rigidity}
Let $R=\Z[\zeta_n]$ for some $n\ge1$, let $m\ge1$, and let
$P$ be an odd prime not dividing $m$. Fix a common block decomposition
of the row and column indices of $M_m(R)$. Suppose
$D_0,\ldots,D_{P-1}\in M_m(R)$ are block diagonal and
$B\in M_m(R)$ has zero diagonal blocks. If
\begin{equation}\label{eq:rigidity-assumptions}
 \sum_{j=0}^{P-1}D_j=mI_m,
 \qquad
 (D_j+B)^2=m(D_j+B)\quad(0\le j<P),
\end{equation}
then $B=0$.
\end{lemma}
\begin{proof}
Summing the quadratic identities in~\eqref{eq:rigidity-assumptions}
and using
\[
 \sum_j(D_jB+BD_j)
 =\left(\sum_jD_j\right)B+B\left(\sum_jD_j\right)=2mB
\]
gives
\[
 \sum_jD_j^2+2mB+PB^2=m^2I_m+mPB.
\]
The matrices $D_j^2$ are block diagonal. Hence the off-diagonal
blocks of this identity give
\begin{equation}\label{eq:off-diagonal}
 P(B^2)_{ii'}=m(P-2)B_{ii'}\qquad(i\ne i'),
\end{equation}
where $(B^2)_{ii'}=\sum_k B_{ik}B_{ki'}$.

We prove by induction that every entry of $B$ belongs to $P^rR$
for every integer $r\ge0$. The assertion for $r=0$ is immediate.
If it holds for $r$, then every entry of $PB^2$ belongs to
$P^{2r+1}R\subseteq P^{r+1}R$. By~\eqref{eq:off-diagonal},
the same is true of $m(P-2)B_{ii'}$ for $i\ne i'$.
Since $m(P-2)$ is coprime to $P$, it has an integer inverse
modulo $P^{r+1}$. Thus every entry of $B_{ii'}$ belongs to
$P^{r+1}R$. The diagonal blocks of $B$ are zero, completing
the induction.

Finally, $R$ is free over $\Z$, with basis
$1,\zeta_n,\ldots,\zeta_n^{\varphi(n)-1}$.
Consequently $\bigcap_{r\ge0}P^rR=\{0\}$, and hence $B=0$.
\end{proof}

\begin{theorem}[Prime-coordinate elimination]\label{thm:eliminate}
Let $n\ge1$ and let $P$ be an odd prime with $P\nmid n$. Suppose
$(A,\Lambda)$ is an $m$-point spectral pair in
$\Z_n\times\Z_P$ and $P\nmid m$.
Then the projections of $A$ and $\Lambda$ onto $\Z_n$ are
injective, and their images form a spectral pair in $\Z_n$.
\end{theorem}
\begin{proof}
For $i,j\in\Z_P$, set
\[
 A_i=\{x:(x,i)\in A\},\qquad
 \Lambda_j=\{u:(u,j)\in\Lambda\}.
\]
Let $F$ be the character matrix of $(A,\Lambda)$, and let $W$
be obtained by deleting the $P$-coordinate factor:
\[
 F_{(x,i),(u,j)}=\zeta_n^{xu}\zeta_P^{ij},
 \qquad W_{(x,i),(u,j)}=\zeta_n^{xu}.
\]
Both are $m\times m$ matrices indexed by the original points;
injectivity of the projections is not assumed. We use the row
block decomposition $\C^A=\bigoplus_i\C^{A_i}$.

Let $F_j,W_j$ consist of the columns indexed by $\Lambda_j$, and put
$Q_j=W_jW_j^*\in M_m(\Z[\zeta_n])$. Multiplying row $(x,i)$
of $W_j$ by $\zeta_P^{ij}$ gives $F_j$. This is a unitary row
scaling, so column orthogonality in $F$ gives
\[
 W_j^*W_j=F_j^*F_j=mI_{|\Lambda_j|}.
\]
It follows that
\begin{equation}\label{eq:quadratic}
 Q_j^2=mQ_j\qquad(0\le j<P).
\end{equation}
An empty column block gives $Q_j=0$.

For rows $(x,i),(x',i')$ in distinct layers, define
\[
 b_j=(Q_j)_{(x,i),(x',i')}
     =\sum_{u\in\Lambda_j}\zeta_n^{(x-x')u}.
\]
Their orthogonality in $F$ reads
\[
 \sum_{j=0}^{P-1}b_j\zeta_P^{(i-i')j}=0.
\]
Lemma~\ref{lem:coefficients} implies $b_0=\cdots=b_{P-1}$.
Thus the matrices $Q_j$ have a common off-diagonal block part $B$,
and we can write
\begin{equation}\label{eq:block-parts}
 Q_j=D_j+B,
\end{equation}
where $D_j$ is block diagonal and $B$ has zero diagonal blocks.
Here block indices refer to the layers $A_i$, so each $B_{ii'}$
is an $|A_i|\times|A_{i'}|$ matrix.
For two rows in the same layer, their $P$-coordinate phases cancel
in the inner product. Row orthogonality and the row norms in $F$
therefore yield
\begin{equation}\label{eq:diagonal-sum}
 \sum_{j=0}^{P-1}D_j=mI_m.
\end{equation}

Equations~\eqref{eq:quadratic}--\eqref{eq:diagonal-sum} verify
all the hypotheses of Lemma~\ref{lem:block-rigidity}, so $B=0$.
Consequently
\[
 WW^*=\sum_jQ_j=\sum_jD_j=mI_m.
\]
The square matrix $W$ is nonsingular. Two points of $A$ with the
same projection would give equal rows of $W$, and two points of
$\Lambda$ with the same projection would give equal columns.
Thus both projections are injective, and $W$ is the character
matrix of their images. These images form an $m$-point spectral pair.
\end{proof}

\begin{corollary}\label{cor:abelian-elimination}
Let $H$ be a finite abelian group of exponent $e$, let $P$ be
an odd prime with $P\nmid e$, and suppose $P\nmid m$.
If $(A,\Lambda)$ is an $m$-point spectral pair in $H\times\Z_P$,
with $\Lambda\subseteq\widehat H\times\Z_P$ under the product
pairing, then the projections of $A$ onto $H$ and of $\Lambda$
onto $\widehat H$ are injective and form a spectral pair.
\end{corollary}
\begin{proof}
Every character of $H$ takes values in $\Z[\zeta_e]$.
Replace $\zeta_n^{xu}$ in the proof of Theorem~\ref{thm:eliminate}
by $\chi(x)$, with $x\in H$ and $\chi\in\widehat H$.
Lemmas~\ref{lem:coefficients} and~\ref{lem:block-rigidity} apply
with $n=e$, and the same argument gives $WW^*=mI_m$.
The conclusions follow from the nonsingularity of $W$.
\end{proof}

\begin{remark}\label{rem:two}
For $P=2$, the coefficient $m(P-2)$ in~\eqref{eq:off-diagonal}
vanishes. The block-matrix lemma therefore requires $P$ to be odd.
In Proposition~\ref{prop:spectral-transversal}, we retain a possible
coordinate of order two and handle it by cardinality and parity.
\end{remark}

\section{Proof of the main theorem}\label{sec:mainproof}

We first prove the spectral implication, which is the new part of
Theorem~\ref{thm:main}.

\begin{proposition}[Spectral sets are transversals]\label{prop:spectral-transversal}
Let $N\ge1$ be square-free. If $A\subseteq\Z_N$ is a nonempty
spectral set of cardinality $m$, then $m\mid N$ and reduction
modulo $m$ maps $A$ bijectively onto $\Z_m$.
\end{proposition}
\begin{proof}
Let $(A,\Lambda)$ be an $m$-point spectral
pair in $\Z_N$, with $N$ square-free. Put $g=\gcd(N,m)$.
Thus $g$ is the product of those prime factors of $N$ that divide
$m$. No divisibility relation between $m$ and $N$ is assumed at
this stage.

Apply Theorem~\ref{thm:eliminate} successively to every odd prime
factor of $N$ that does not divide $m$. Each step preserves the
cardinality and spectrality of both projected sets and is injective
on them. More explicitly, after deleting coordinates whose product
is $d$, the spatial and frequency labels in the current cyclic
pairing are $a\bmod(N/d)$ and $d^{-1}\lambda\bmod(N/d)$,
respectively. Thus the spatial map remains the natural quotient
throughout the iteration. The final spectral pair lies in a group of order
\[
 M=\begin{cases}
 2g,&N\text{ even and }m\text{ odd},\\
 g,&\text{otherwise}.
 \end{cases}
\]
Hence $g\mid m\le M$. If $M=g$, this gives $m=g$.
If $M=2g$, then $m/g$ is an odd positive integer at most $2$,
so again $m=g$. In particular, $m\mid N$.

When $M=m$, the final projected set is the whole group $\Z_m$.
Suppose now that $M=2m$, so $N$ is even and $m$ is odd. The final pair, say
$(A',\Lambda')$, lies in $\Z_m\times\Z_2$.
We use the product pairing, relabelling the dual coordinates as
at the beginning of Section~\ref{sec:elimination}. The coordinate
projection onto $\Z_m$ is injective on $A'$. Indeed, a collision
would give two points $(x,0),(x,1)\in A'$, whose row orthogonality
would imply
\[
 0=\sum_{(u,j)\in\Lambda'}(-1)^j
   =|\Lambda'_0|-|\Lambda'_1|.
\]
This makes $m=|\Lambda'|$ even, a contradiction. Since $|A'|=m$,
its projection is the whole of $\Z_m$.

In all cases, $m\mid N$, and the composition of these CRT
projections is the natural quotient map $\Z_N\to\Z_m$ on the
spatial group. It is bijective on $A$. This proves the proposition.
\end{proof}

We next recall a dilation lemma for finite tiles. It is due to
Tijdeman~\cite[Theorem~1]{Tijdeman}; its finite polynomial form
and the Frobenius proof used here appear in
Coven--Meyerowitz~\cite[Lemma~3.1]{CM}. See also
Shi~\cite[Lemma~4.2]{Shi} for the finite-group formulation and
\cite[Proposition~4.1]{Shi} for the resulting square-free
transversal characterization.

\begin{lemma}[Finite dilation]\label{lem:dilation}
Suppose $C\oplus D=\Z_N$ and $|C|=k$. If $r$ is prime and
$r\nmid k$, then the map $c\mapsto rc$ is injective on $C$ and
$rC\oplus D=\Z_N$.
\end{lemma}
\begin{proof}
Work in the commutative group ring
$\Z[\Z_N]\simeq\Z[X]/(X^N-1)$, so that coefficients are indexed
by residues modulo $N$. Write
$C(X)=\sum_{c\in C}X^c$, and similarly for $D$, and put
$U=\sum_{x\in\Z_N}X^x$. The tiling identity is $C(X)D(X)=U$.
For every group-ring element $H$ one has $HU=H(1)U$.
Reducing coefficients modulo $r$ and applying Frobenius gives
\[
 C(X^r)D(X)
 \equiv C(X)^rD(X)
 =C(X)^{r-1}U
 =k^{r-1}U
 \equiv U\pmod r.
\]
The last congruence uses Fermat's theorem and $r\nmid k$.
The coefficients of $C(X^r)D(X)$ in the integral group ring are
nonnegative integers, each congruent to $1$ modulo $r$, and hence
each at least $1$. They sum to $k|D|=N$,
and there are $N$ coefficients. Thus every coefficient is $1$.
This also excludes repeated points in the list $rC$, since any
such repetition would give a coefficient at least $2$ after
addition of an element of the nonempty set $D$.
\end{proof}

\begin{proof}[Proof of Theorem~\ref{thm:main}]
Proposition~\ref{prop:spectral-transversal} proves
\textup{(1)}$\Rightarrow$\textup{(3)}.

To prove \textup{(2)}$\Rightarrow$\textup{(3)}, suppose that
$A\oplus B=\Z_N$ and $|A|=m$. Counting gives
$m\mid N$. Put $b=N/m$. Since $N$ is square-free,
$\gcd(b,m)=1$. Apply Lemma~\ref{lem:dilation} successively to
the prime factors of $b$. The dilation $a\mapsto ba$ is injective
on $A$, and $bA$ has $m$ elements. But it is contained in the
subgroup $b\Z_N$, also of order $m$, so
\[
 bA=b\Z_N.
\]
The equality $ba=ba'$ in $\Z_N$ is equivalent to $a=a'$ modulo
$m$. Thus reduction modulo $m$ maps $A$ bijectively onto $\Z_m$.

Finally, suppose \textup{(3)} holds. The unique representative
of each residue modulo $m$ immediately gives
$A\oplus m\Z_N=\Z_N$. To see spectrality directly, set
\[
 \Lambda_0=\left\{\frac Nm j:0\le j<m\right\}\subseteq\Z_N.
\]
For distinct $j,k\in\{0,\ldots,m-1\}$,
\[
 \sum_{a\in A}e^{2\pi i a(j-k)/m}
   =\sum_{u=0}^{m-1}e^{2\pi i u(j-k)/m}=0.
\]
The $m$ columns indexed by $\Lambda_0$ have squared norm $m$,
so $\Lambda_0$ is a spectrum. This completes the proof of
Theorem~\ref{thm:main}.
\end{proof}

\end{document}